\documentclass[11pt,a4paper]{article}
\usepackage[affil-it]{authblk}
\usepackage[utf8]{inputenc}
\usepackage{amsmath, amssymb, mathrsfs, amsthm, amsfonts}
\usepackage{latexsym}
\usepackage{hyperref}
\usepackage{a4wide}
\usepackage{color}
\usepackage{graphicx}

\newtheorem{theorem}{Theorem}[section]
\newtheorem{corollary}[theorem]{Corollary}
\newtheorem{lemma}[theorem]{Lemma}
\newtheorem{proposition}[theorem]{Proposition}

\theoremstyle{definition}
\newtheorem{definition}[theorem]{Definition}
\theoremstyle{remark}
\newtheorem{example}{Example}

\title{The Varchenko Matrix for Dehyperplane Arrangement}

\author{Hery Randriamaro
\thanks{This research was funded by my mother \\
Lot II B 32 bis Faravohitra, 101 Antananarivo, Madagascar \\
e-mail: \texttt{hery.randriamaro@outlook.com}}}

\begin{document}

\maketitle

\begin{abstract}
\noindent This article computes the Varchenko determinant of dehyperplane arrangements which are generalizations of pseudohyperplane arrangements. But unlike those latter, they are defined on a real manifold, and it is not always possible to obtain a central dehyperplane arrangement by coning. This article also studies the solution space of a linear system defined from a dehyperplane arrangement. That equation system was first introduced by Aguiar and Mahajan for central hyperplane arrangements.  

\bigskip 


\smallskip

\end{abstract}

\section{Introduction}

\noindent For $n \in \mathbb{N}$, let $\mathscr{R}_n$ be the category of nonempty real open contractible $n$-dimensional manifolds. A dehyperplane in $T \in \mathscr{R}_n$ is a set $H \subseteq T$ such that $H \in \mathscr{R}_{n-1}$ and $H$ cuts $T$ into two elements of $\mathscr{R}_n$. For a dehyperplane $H$ in $T$, denote by $H^+, H^-$ both elements of $\mathscr{R}_n$ such that $H^+ \sqcup H \sqcup H^- = T$ and $\overline{H^+} \cap \overline{H^-} = H$. Consider a finite set $\mathcal{A}$ of dehyperplanes in $T$. A flat of $\mathcal{A}$ is a nonempty intersection of dehyperplanes in $\mathcal{A}$. Denote by $L_{\mathcal{A}}$ the set formed by the flats of $\mathcal{A}$, the intersection of element in $\emptyset$ being $T$. It is a meet semilattice with partial order $\leq$ defined, for $X,Y \in L_{\mathcal{A}}$, by $X \leq Y \Longleftrightarrow Y \subseteq X$. The set $\mathcal{A}$ is a dehyperplane arrangement in $T$ if, for every $H \in \mathcal{A}$ and $X \in L_{\mathcal{A}}$ such that $H \cap X \neq \emptyset$, we have
\begin{itemize}
	\item either $H \leq X$,
	\item or $\exists i \in [n]:\, H \cap X \in \mathscr{R}_{i-1},\ H^+ \cap X \in \mathscr{R}_i,\ H^- \cap X \in \mathscr{R}_i$.
\end{itemize}

\noindent Letting $H^0 := H$ for a dehyperplane $H$, a face of $\mathcal{A}$ is a nonempty subset $F \subseteq T$ having the form $\displaystyle F= \bigcap_{H \in \mathcal{A}} H^{\epsilon_H(F)}$ with $\epsilon_H(F) \in \{+,0,-\}.$ Denote $F_{\mathcal{A}}$ the set formed by the faces of $\mathcal{A}$. It is a poset with partial order $\preceq$ defined, for $F, G \in F_{\mathcal{A}}$, by $$F \preceq G \quad \Longleftrightarrow \quad \forall H \in \mathcal{A}:\ \epsilon_H(F) \in \big\{0, \epsilon_H(G)\big\}.$$

\noindent The sign sequence of a face $F \in F_{\mathcal{A}}$ is	$\epsilon_{\mathcal{A}}(F) := \big(\epsilon_H(F)\big)_{H \in \mathcal{A}}$. A chamber of $\mathcal{A}$ is a face whose sign sequence contains no $0$. Denote the set formed by the chambers of $\mathcal{A}$ by $C_{\mathcal{A}}$.

\smallskip

\noindent The dehyperplane arrangements were recently introduced, and their $f$-polynomial computed \cite{Ra3}. As mentioned at the end of that article, we think that it is possible to provide a generalization of the topological representation theorem by using dehyperplane arrangements. In other words, we believe that every conditional oriented matroid is poset isomorphic to $F_{\mathcal{A}}$ for some dehyperplane arrangement $F_{\mathcal{A}}$. The former was introduced by Bandelt et al. \cite{BaChKn}, and is a common generalization of oriented matroids and lopsided sets.

\begin{definition}
Let $\mathcal{K}$ be a subset of a dehyperplane arrangement $\mathcal{A}$ in $T \in \mathscr{R}_n$. An \textbf{apartment} of $\mathcal{A}$ is a chamber of $\mathcal{K}$. Denote the set formed by the apartments of $\mathcal{A}$ by $K_{\mathcal{A}}$.
\end{definition}

\noindent The sets formed by the faces and the chambers in an apartment $K \in K_{\mathcal{A}}$ are respectively $$F_{\mathcal{A}}^K := \{F \in F_{\mathcal{A}}\ |\ F \subseteq K\} \quad \text{and} \quad C_{\mathcal{A}}^K := C_{\mathcal{A}} \cap F_{\mathcal{A}}^K.$$

\noindent For $H \in \mathcal{A}$ and $\varepsilon \in \{+,-\}$, assign a variable $q_H^{\varepsilon}$ to every open spaces $H^{\varepsilon}$. We work with the polynomial ring $R_{\mathcal{A}} := \mathbb{Z}\big[q_H^{\varepsilon}\ \big|\ \varepsilon \in \{+,-\},\, H \in \mathcal{A}\big]$. For $C,D \in C_{\mathcal{A}}$, the set of formed by the open spaces containing $C$ but not $D$ is $\mathscr{H}(C,D) := \big\{H^{\epsilon_H(C)}\ \big|\ H \in \mathcal{A},\, \epsilon_H(C) = - \epsilon_H(D)\big\}$. Define an extension $\mathrm{v}: C_{\mathcal{A}} \times C_{\mathcal{A}} \rightarrow R_{\mathcal{A}}$ to dehyperplane arrangements of the distance function of Aguiar and Mahajan \cite[§~8.1]{AgMa} by
$$\mathrm{v}(C,C) = 1 \quad \text{and} \quad \mathrm{v}(C,D) = \prod_{H^{\varepsilon} \in \mathscr{H}(C,D)} q_H^{\varepsilon}\,\ \text{if}\,\ C \neq D.$$

\begin{definition}
The \textbf{Varchenko matrix} for an apartment $K$ of a dehyperplane arrangement $\mathcal{A}$ in $T \in \mathscr{R}_n$ is $V_{\mathcal{A}}^K :=\big|\mathrm{v}(D,C)\big|_{C,D \in C_{\mathcal{A}}^K}$.
\end{definition}

\noindent For a dehyperplane arrangement $\mathcal{A}$ in $T$, we just write $V_{\mathcal{A}}$ for $V_{\mathcal{A}}^T$. That matrix was originally defined for hyperplane arrangements in $\mathbb{R}^n$ and with the restriction $q_H^+ = q_H^-$ by Varchenko \cite[§ 1]{Va}. But it already appeared earlier in the implicit form of a symmetric matrix for a Verma module over a $\mathbb{C}$-algebra \cite[§ 1]{ScVa}. It plays a key role to prove the realizability of variant models of quon algebras like the multiparametric quon algebra \cite[Proposition~2.1]{Ra2} in quantum statistics. Moreover, algebraic structures of the Varchenko matrix have been studied over time. Gao and Zhang computed its diagonal form \cite[Theorem~2]{GaZh} for hyperplane arrangements in semigeneral position with the same restriction. Then, for $q_H^+ = q_H^- = q$, Denham and Hanlon studied its Smith normal form \cite[Theorem~3.3]{DeHa}, and Hanlon and Stanley computed the nullspace of the Varchenko matrix of braid arrangements \cite[Theorem~3.3]{HaSt}.

\begin{definition}
The \textbf{centralization} of a dehyperplane arrangement $\mathcal{A}$ in $T \in \mathscr{R}_n$ to a face $F \in F_{\mathcal{A}} \setminus C_{\mathcal{A}}$ is the dehyperplane arrangement $\mathcal{A}_F := \{H \in \mathcal{A}\ |\ F \subseteq H\}$ in $T$. The \textbf{weight} and \textbf{multiplicity} of $F$ are respectively the monomial and integer
$$\mathrm{b}_F := \prod_{H \in \mathcal{A}_F} q_H^+ q_H^- \quad \text{and} \quad \beta_F := \frac{\#\{C \in C_{\mathcal{A}}\ |\ \overline{C} \cap H = F\}}{2},$$
where $H \in \mathcal{A}_F$, and we will see at the end of Section~\ref{SeDet} that $\beta_F$ is independent of $H$.
\end{definition}

\noindent We can now state the first main result of this article.

\begin{theorem} \label{ThDet}
Let $\mathcal{A}$ be a dehyperplane arrangement in $T \in \mathscr{R}_n$, and $K \in K_{\mathcal{A}}$. Then,  $$\det V_{\mathcal{A}}^K = \prod_{F \in F_{\mathcal{A}}^K \setminus C_{\mathcal{A}}^K} (1 - \mathrm{b}_F)^{\beta_F}.$$
\end{theorem}

\noindent It is the Varchenko determinant of the dehyperplane arrangement $\mathcal{A}$ for the apartment $K$.

\begin{corollary} \label{CoTh}
For a dehyperplane arrangement $\mathcal{A}$ in $T \in \mathscr{R}_n$, we have  $$\det V_{\mathcal{A}} = \prod_{F \in F_{\mathcal{A}} \setminus C_{\mathcal{A}}} (1 - \mathrm{b}_F)^{\beta_F}.$$
\end{corollary}

\noindent The Varchenko determinant has known several investigations over time. The original computing was that of Varchenko for hyperplane arrangements with the restriction $q_H^+ = q_H^-$ \cite[Theorem~1.1]{Va}. Still with that restriction, Gente computed that determinant for cone of hyperplane arrangements \cite[Theorem~4.5]{Ge}. Then, Aguiar and Mahajan computed that determinant for central hyperplane arrangements \cite[Theorem~8.11]{AgMa} and their cones \cite[Theorem~8.12]{AgMa} using the distance function $\mathrm{v}$. Two recent results are the Varchenko determinant for oriented matroids with $q_H^+ = q_H^-$ computed by Hochstättler and Welker \cite[Theorem~1]{HoWe}, and that for pseudohyperplane arrangements we computed \cite[Theorem~1.5]{Ra4}. The topological representation theorem links both results as it states that every oriented matroid is poset isomorphic to $F_{\mathcal{A}}$ for some central pseudohyperplane arrangement $\mathcal{A}$ \cite[Theorem~3.5, Corollary~3.13]{De}. Besides, as mentioned at the end of their article, the referee suggested Hochstättler and Welker conditional oriented matroids as possible direction to generalize \cite[Theorem~1]{HoWe}. If a generalized topological representation theorem linking conditional oriented matroids with dehyperplane arrangements is proved, then Corollary~\ref{CoTh} would be the generalization proposed by that referee. Furthermore, all those cited Varchenko determinants have a common point: it suffices to investigate central hyperplane or pseudohyperplane arrangements to obtain the Varchenko determinant of arbitrary ones by using the coning described in \cite[§~4]{Ra4}. That is not the case for dehyperplane arrangements. The following example shows one for which the Varchenko determinant cannot be computed by coning.

\begin{example}
Consider the dehyperplane arrangement $\mathcal{A}_{ex} = \{P_1, P_2, P_3, P_4\}$ in $\mathbb{R}^2$ represented in Figure \ref{Ex1}. Assigning the variable $q_i^+$ resp. $q_i^-$ to the set $P_i^+$ resp. $P_i^-$ where $i \in [4]$, we get
$$\displaystyle \det \begin{pmatrix}
1 & q_1^+ q_2^- & q_1^+ q_3^- & q_1^+ q_4^- & q_1^+ \\
q_1^- q_2^+ & 1 & q_2^+ q_3^- & q_2^+ q_4^- & q_2^+ \\
q_1^- q_3^+ & q_2^- q_3^+ & 1 & q_3^+ q_4^- & q_3^+ \\
q_1^- q_4^+ & q_2^- q_4^+ & q_3^- q_4^+ & 1 & q_4^+ \\
q_1^- & q_2^- & q_3^- & q_4^- & 1 
\end{pmatrix} = \prod_{i \in [4]} (1 - q_i^+ q_i^-).$$
	
\begin{figure}[h]
	\centering
	\includegraphics[scale=0.73]{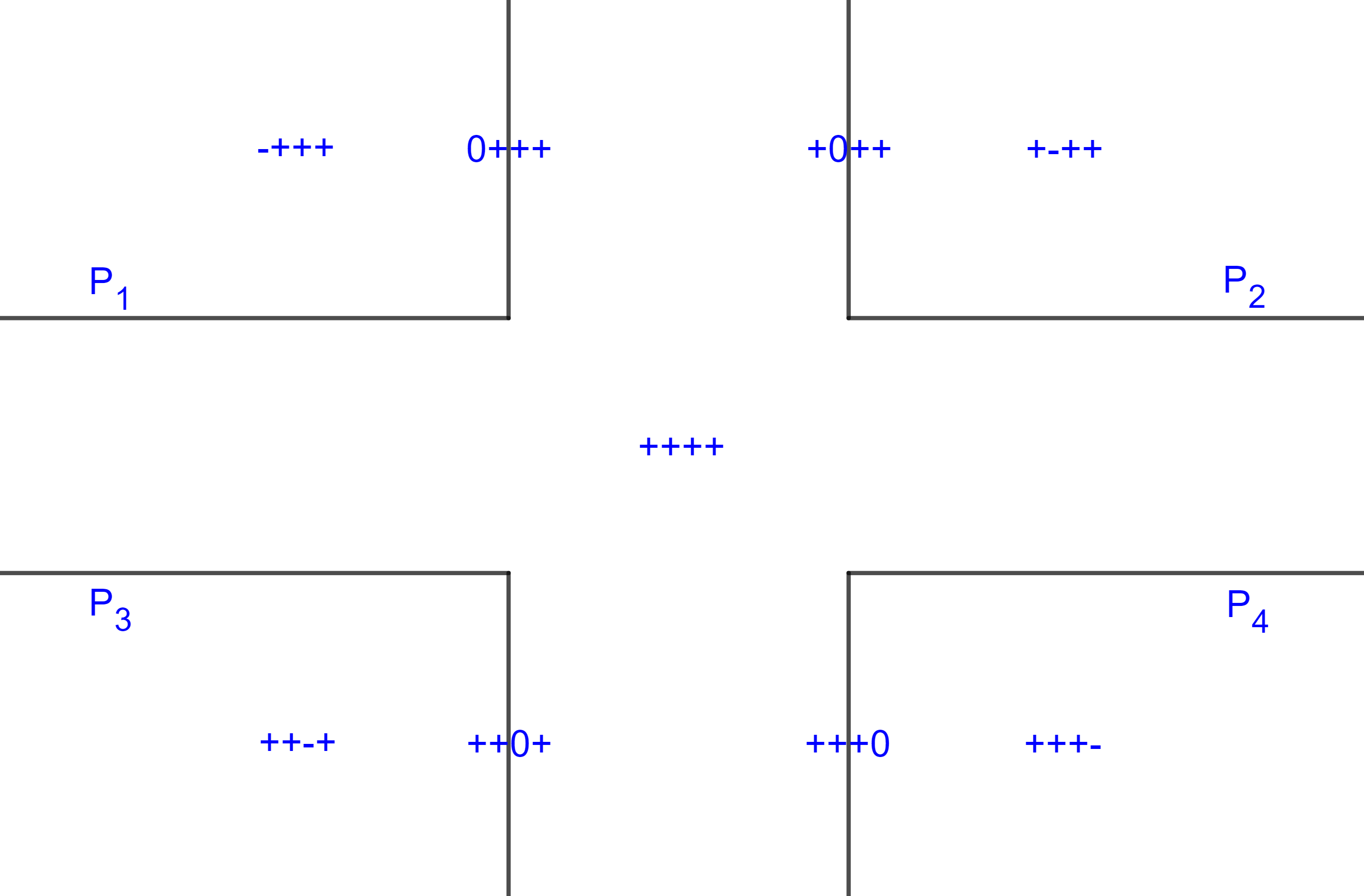}
	\caption{The Dehyperplane Arrangement $\mathcal{A}_{ex}$ in $\mathbb{R}^2$}
	\label{Ex1}
\end{figure}
\end{example}

\noindent We know that the restriction $\mathcal{A}^X := \{H \cap X \in L_{\mathcal{A}}\ |\ H \in \mathcal{A},\, X \nsubseteq H,\, H \cap X \neq \emptyset\}$ of a dehyperplane arrangement $\mathcal{A}$ on a flat $X \in L_{\mathcal{A}}$ is a dehyperplane arrangement in $X \in \mathscr{R}_{\dim X}$ \cite[Lemma~2.2]{Ra3}. The support of a face $F \in F_{\mathcal{A}}$ is the subset $\mathrm{s}(F) := \{H \in \mathcal{A}\ |\ \epsilon_H(F)=0\}$ of $\mathcal{A}$. The sets formed by the faces and the chambers of $\mathcal{A}^X$ are respectively $$F_{\mathcal{A}^X} := \{F \in F_{\mathcal{A}}\ |\ \mathrm{s}(F) \subseteq X\} \quad \text{and} \quad C_{\mathcal{A}^X} := \{F \in F_{\mathcal{A}}\ |\ \mathrm{s}(F) = X\}.$$

\noindent For $C,D \in C_{\mathcal{A}^X}$, let $\mathscr{H}(C,D) := \big\{H^{\epsilon_H(C)}\ \big|\ H \in \mathcal{A},\, H \cap X \in \mathcal{A}^X,\, \epsilon_H(C) = - \epsilon_H(D)\big\}$. Define $\mathrm{v}^X: C_{\mathcal{A}}^X \times C_{\mathcal{A}}^X \rightarrow R_{\mathcal{A}}$ by
$\mathrm{v}^X(C,C) = 1$ and $\displaystyle \mathrm{v}^X(C,D) = \prod_{H^{\varepsilon} \in \mathscr{H}(C,D)} q_H^{\varepsilon}$ if $C \neq D$.

\noindent We will see in Section~\ref{SeSe} that $F_{\mathcal{A}}$ forms a semigroup together with the binary operation defined as follows: If $F,G \in F_{\mathcal{A}}$, then $FG$ is the face in $F_{\mathcal{A}}$ such that, for every $H \in \mathcal{A}$, $$\displaystyle \epsilon_H(FG) := \begin{cases}
\epsilon_H(F) & \text{if}\ \epsilon_H(F) \neq 0, \\ \epsilon_H(G) & \text{otherwise} \end{cases}.$$

\noindent Extend the distance functions $\mathrm{v}^X$ to $\mathbf{v}: F_{\mathcal{A}} \times F_{\mathcal{A}} \rightarrow R_{\mathcal{A}}$, for $F, G \in F_{\mathcal{A}}$, by $$\displaystyle \mathbf{v}(F,G) := \mathrm{v}^{\mathrm{s}(FG)}(FG,GF).$$

\noindent The set formed by the minimal elements of $F_{\mathcal{A}}$ is $\min F_{\mathcal{A}} := \{F \in F_{\mathcal{A}}\ |\ \nexists G \in F_{\mathcal{A}}:\, G \prec F\}$. Assign a variable $x_F$ to each face $F \in F_{\mathcal{A}}$. 

\begin{definition}
Let $\mathcal{A}$ be a dehyperplane arrangement in $T \in \mathscr{R}_n$. The \textbf{Aguiar-Mahajan system} for $\mathcal{A}$ is the linear equation system $$\sum_{\substack{F \in F_{\mathcal{A}} \\ GF = G}} x_F\,\mathbf{v}(F,G) = 0 \quad \text{indexed by} \quad G \in F_{\mathcal{A}} \setminus \min F_{\mathcal{A}}.$$
\end{definition}

\noindent That system was introduced, and solved by Aguiar and Mahajan for central hyperplane arrangements \cite[Theorem~8.19]{AgMa}. A dehyperplane arrangement $\mathcal{A}$ in $T$ is said central if $\displaystyle \bigcap_{H \in \mathcal{A}} H \neq \emptyset$. In that case,  we will see in Lemma~\ref{LeCe} that, for every $F \in F_{\mathcal{A}}$, there exists $\tilde{F} \in F_{\mathcal{A}}$ such that $\epsilon_{\mathcal{A}}(F) = -\epsilon_{\mathcal{A}}(\tilde{F})$. Here is the second main result of this article.

\begin{theorem} \label{ThSy}
Let $\mathcal{A}$ be a dehyperplane arrangement in $T \in \mathscr{R}_n$. The solution space dimension of the Aguiar-Mahajan system of $\mathcal{A}$ is $\# \min F_{\mathcal{A}}$. In the particular case that $\mathcal{A}$ is central, then $\min F_{\mathcal{A}} = \{O\}$, and starting with an arbitrary value of $x_O$, there is a unique solution which can be computed recursively with the formula
$$x_G = \frac{-1}{1 - \mathbf{v}(G,\tilde{G}) \, \mathbf{v}(\tilde{G},G)} \sum_{\substack{F \in F_{\mathcal{A}} \\ F \prec G}} \big(x_F + (-1)^{\mathrm{rk}\,G}x_{\tilde{F}} \, \mathbf{v}(\tilde{G},G)\big).$$
\end{theorem}

\noindent This article is structured as follows: We prove in Section~\ref{SeSe} that $F_{\mathcal{A}}$ forms a semigroup together with the binary operation defined above. Then, we extend Witt identities to dehyperplane arrangements in Section~\ref{SeWi}. Those extensions are used at the end to compute $\det V_{\mathcal{A}}^K$ in Section~\ref{SeDet}, and to prove Theorem~\ref{ThSy} in Section~\ref{SeSy}.

\section{A Face Semigroup for Dehyperplane Arrangement} \label{SeSe}

\noindent We prove that $F_{\mathcal{A}}$ forms a semigroup with the operation $FG$ for $F,G \in F_{\mathcal{A}}$.

\begin{proposition} \label{PrSG}
Let $\mathcal{A}$ be a dehyperplane arrangement in $T \in \mathscr{R}_n$, and $F,G \in F_{\mathcal{A}}$. Then, there exists a unique face $E \in F_{\mathcal{A}}$ such that
$$\displaystyle \forall H \in \mathcal{A}:\,   \epsilon_H(E) = \begin{cases}
\epsilon_H(F) & \text{if}\ \epsilon_H(F) \neq 0, \\ \epsilon_H(G) & \text{otherwise} \end{cases}.$$
\end{proposition}

\begin{proof}
Consider the subset $\mathcal{K} = \big\{H \in \mathcal{A}\ \big|\ \epsilon_H(F) \neq 0\big\}$, and the apartment $\displaystyle K = \bigcap_{H \in \mathcal{K}} H^{\epsilon_H(F)}$. If $\mathcal{K} = \mathcal{A}$, then $E = F$. Otherwise, we have $F \varsubsetneq K$, therefore for every $L \in F_{\mathcal{A}_F}$, $L \cap K \neq \emptyset$. Note that the set $\displaystyle \bigcap_{H \in \mathcal{A}_F} H^{\epsilon_H(G)}$ is nonempty since it contains $G$. As that set is a face of $\mathcal{A}_F$, we consequently obtain $\displaystyle E = K \cap \bigcap_{H \in \mathcal{A}_F} H^{\epsilon_H(G)}$.
\end{proof}

\begin{corollary} \label{CoSemi}
Given a dehyperplane arrangement $\mathcal{A}$ in $T \in \mathscr{R}_n$, the set $F_{\mathcal{A}}$ forms a semigroup together with the binary operation defined by: If $F,G \in F_{\mathcal{A}}$, then $FG$ is the face in $F_{\mathcal{A}}$ such that, for every $H \in \mathcal{A}$, $\displaystyle \epsilon_H(FG) := \begin{cases}
\epsilon_H(F) & \text{if}\ \epsilon_H(F) \neq 0, \\ \epsilon_H(G) & \text{otherwise} \end{cases}$.
\end{corollary}

\begin{proof}
It remains to prove the associativity of the binary operation. Let $E,F,G \in F_{\mathcal{A}}$, and $H \in \mathcal{A}$. Then,
$\displaystyle \epsilon_H\big((EF)G\big) =\ \begin{cases}
\epsilon_H(E) & \text{if}\ \epsilon_H(E) \neq 0, \\
\epsilon_H(F) & \text{if}\ \epsilon_H(E)=0\ \text{and}\ \epsilon_H(F) \neq 0, \\
\epsilon_H(G) & \text{otherwise} \end{cases}\ = \epsilon_H\big(E(FG)\big)$.	
\end{proof}

\noindent It is known that, if $\mathcal{A}$ is a central hyperplane arrangement , then $F_{\mathcal{A}}$ together with that binary operation is the Tits monoid \cite[§ 1.4.2]{AgMa}. Besides, if the generalized topological representation theorem stated in the introduction is proved, one could immediately conclude Corollary~\ref{CoSemi} from the fact that a conditional oriented matroid is a semigroup \cite[Proposition~2.12]{MaSaSt}.

\section{Witt Identities on Dehyperplane Arrangement} \label{SeWi}

\noindent We extend Witt identities to dehyperplane arrangements. They will be used later to compute $\det V_{\mathcal{A}}^K$, and to investigate the Aguiar-Mahajan system.

\begin{lemma} \label{LeCe}
If $\mathcal{A}$ is a central dehyperplane arrangement in $T \in \mathscr{R}_n$, then
$$\forall F \in F_{\mathcal{A}},\, \exists \tilde{F} \in F_{\mathcal{A}}:\ \forall H \in \mathcal{A}:\, \epsilon_H(\tilde{F}) = - \epsilon_H(F).$$
\end{lemma}

\begin{proof}
Let $F \in F_{\mathcal{A}} \setminus C_{\mathcal{A}}$, and define a path $\displaystyle p:[0,1] \rightarrow \bigcup_{H \in \mathcal{A}_F}H$ as follows: $p$ starts at a point $p(0) \in F$, crosses all the hyperplanes $H \in \mathcal{A} \setminus \mathcal{A}_F$, and ends at the $p(1)$ after those crossings. Then, $\tilde{F}$ is the face of $\mathcal{A}$ containing $p(1)$. Now, considering a chamber $C \in C_{\mathcal{A}}$, $\tilde{C}$ is the chamber of $\mathcal{A}$ such that $\displaystyle \partial \tilde{C} = \bigsqcup_{\substack{F \in F_{\mathcal{A}} \setminus C_{\mathcal{A}} \\ F \varsubsetneq \partial C}} \tilde{F}$.
\end{proof}

\begin{definition}
A \textbf{nested face} of a dehyperplane arrangement $\mathcal{A}$ in $T \in \mathscr{R}_n$ is a pair $(F,G)$ of faces in $F_{\mathcal{A}}$ such that $F \prec G$.
\end{definition}

\noindent For a nested face $(F,G)$ of $\mathcal{A}$, let $F_{\mathcal{A}}^{(F,G)}$ be the set of faces $\{L \in F_{\mathcal{A}}\ |\ F \preceq L \preceq G\}$.

\begin{proposition}
Let $\mathcal{A}$ be a dehyperplane arrangement in $T \in \mathscr{R}_n$, $D \in C_{\mathcal{A}}$, and $(A,D)$ a nested face of $\mathcal{A}$. Then, $C_{\mathcal{A}}$ has a chamber $\tilde{D}_A$ whose sign sequence is defined by
$$\forall H \in \mathcal{A}:\,  \epsilon_H(\tilde{D}_A) = \begin{cases} -\epsilon_H(D) & \text{if}\ \epsilon_H(A) = 0, \\
\epsilon_H(A) & \text{otherwise}. \end{cases}$$
\end{proposition}

\begin{proof}
Consider the apartment $\displaystyle K = \bigcap_{\substack{H \in \mathcal{A} \\ \epsilon_H(A) \neq 0}} H^{\epsilon_H(A)}$. We have $C_{\mathcal{A}}^K = \{C \cap K\ |\ C \in C_{\mathcal{A}_A}\}$. It is clear that $D \in C_{\mathcal{A}}^K$. Moreover, we know from Lemma~\ref{LeCe} that $C_{\mathcal{A}}^K$ has a chamber $\tilde{D}_A$ such that, for every $H \in \mathcal{A}_A$, $\epsilon_H(\tilde{D}_A) = -\epsilon_H(D)$, which is consequently the sought chamber.
\end{proof}

\begin{definition}
Let $\mathcal{A}$ be a dehyperplane arrangement in $T \in \mathscr{R}_n$, and define the integer $c_{\mathcal{A}} := \min \, \{\dim F\ |\ F \in F_{\mathcal{A}}\}$. The \textbf{rank} of a face $F \in F_{\mathcal{A}}$ is
$\mathrm{rk}\,F := \dim F - c_{\mathcal{A}}$.
\end{definition}

\noindent For a dehyperplane arrangement $\mathcal{A}$ in $T$ and $D \in C_{\mathcal{A}}$, denote by $F_{\mathcal{A}}^{n-1,D}$ the set of faces $\{F \in F_{\mathcal{A}}\ |\ F \preceq D,\, \dim F = n-1\}$. Besides, let $\chi$ be the function Euler characteristic of the structure of a topological space.

\begin{proposition}  \label{Witt}
Let $\mathcal{A}$ be a dehyperplane arrangement in $T \in \mathscr{R}_n$, $D \in C_{\mathcal{A}}$, and $(A,D)$ a nested face of $\mathcal{A}$. Then, $$\sum_{F \in F_{\mathcal{A}}^{(A,D)}} (-1)^{\mathrm{rk}\,F} \sum_{\substack{C \in C_{\mathcal{A}} \\ FC = D}} x_C \, = \, (-1)^{\mathrm{rk}\,D} \sum_{\substack{C \in C_{\mathcal{A}} \\ AC = \tilde{D}_A}} x_C.$$
\end{proposition}

\begin{proof}
The proof is the extension of \cite[Proposition~4.2]{Ra1} to dehyperplane arrangements. We have $$\displaystyle \sum_{F \in F_{\mathcal{A}}^{(A,D)}} (-1)^{\mathrm{rk}\,F} \sum_{\substack{C \in C_{\mathcal{A}} \\ FC = D}} x_C = \sum_{C \in C_{\mathcal{A}}} \Big( \sum_{\substack{F \in F_{\mathcal{A}}^{(A,D)} \\ FC = D}} (-1)^{\mathrm{rk}\,F} \Big) x_C.$$ 
\begin{itemize}
\item If $\epsilon_{\mathcal{A}_A}(C) = \epsilon_{\mathcal{A}_A}(\tilde{D}_A)$, then $\displaystyle \sum_{\substack{F \in F_{\mathcal{A}}^{(A,D)} \\ FC = D}} (-1)^{\mathrm{rk}\,F} = (-1)^{\mathrm{rk}\,D}$.
\end{itemize}

\noindent Denote by $F_{\mathcal{A}}^{A,n}$ the set of faces $\{F \in F_{\mathcal{A}}\ |\ A \preceq F\} = \big\{F \in F_{\mathcal{A}}\ \big|\ \epsilon_{\mathcal{A} \setminus \mathcal{A}_A}(F) = \epsilon_{\mathcal{A} \setminus \mathcal{A}_A}(A)\big\}$. Let $f: F_{\mathcal{A}}^{A,n} \rightarrow F_{\mathcal{A}_A}$ be the bijection such that, if $F \in F_{\mathcal{A}}^{A,n}$, then $f(F)$ is the face of $\mathcal{A}_A$ having the property
$$\forall H \in \mathcal{A}_A:\, \epsilon_H\big(f(F)\big) = \epsilon_H(F).$$

\begin{itemize}
\item If $\epsilon_{\mathcal{A}_A}(C) = \epsilon_{\mathcal{A}_A}(D)$, then
$$\sum_{\substack{F \in F_{\mathcal{A}}^{(A,D)} \\ FC = D}} (-1)^{\mathrm{rk}\,F} = (-1)^{-c_{\mathcal{A}}} \sum_{F \in f\big(F_{\mathcal{A}}^{(A,D)}\big)} (-1)^{\dim F} = (-1)^{-c_{\mathcal{A}}} \chi\big(\overline{f(D)}\big) = 0.$$
\item The case $\epsilon_{\mathcal{A}_A}(C) \notin \big\{\epsilon_{\mathcal{A}_A}(D), \epsilon_{\mathcal{A}_A}(\tilde{D}_A)\big\}$ remains. Assume $\epsilon_H(D) = +$ for $H \in \mathcal{A}_A$, and define the dehyperplane arrangement $\mathcal{A}_A(C) := \big\{H \in \mathcal{A}_A\ \big|\ \epsilon_H(C) = -\big\}$. If $\#\mathcal{A}_A(C) > 1$ and $E \in F_{\mathcal{A}_A(C)}$, then $$\forall F \in F_{\mathcal{A}_A(C)}^{n-1,E},\, \exists F' \in F_{\mathcal{A}_A(C)}^{n-1,E} \setminus \{F\}:\ \mathrm{int}(\overline{F} \cap \overline{F'}) \in \mathscr{R}_{n-2}.$$
We obtain,
$$\sum_{\substack{F \in F_{\mathcal{A}}^{(A,D)} \\ FC = D}} (-1)^{\mathrm{rk}\,F} = \sum_{\substack{F \in f\big(F_{\mathcal{A}}^{(A,D)}\big) \\ \forall H \in \mathcal{A}_A(C):\ \epsilon_H(F) = +}} (-1)^{\mathrm{rk}\,F} = (-1)^{-c_{\mathcal{A}}} \chi\Big(\overline{f(D)} \setminus \bigcup_{F \in F_{\mathcal{A}_A(C)}^{n-1,f(D)}} \overline{F}\Big) = 0.$$
\end{itemize}
So $\displaystyle \sum_{F \in F_{\mathcal{A}}^{(A,D)}} (-1)^{\mathrm{rk}\,F} \sum_{\substack{C \in C_{\mathcal{A}} \\ FC = D}} x_C \, = \, (-1)^{\mathrm{rk}\,D} \sum_{\substack{C \in C_{\mathcal{A}} \\ \epsilon_{\mathcal{A}_A}(C) \, = \, \epsilon_{\mathcal{A}_A}(\tilde{D}_A)}} x_C \, = \, (-1)^{\mathrm{rk}\,D} \sum_{\substack{C \in C_{\mathcal{A}} \\ AC = \tilde{D}_A}} x_C$.
\end{proof}

\begin{corollary} \label{CoWi}
	Let $\mathcal{A}$ be a dehyperplane arrangement in $T \in \mathscr{R}_n$, $G \in F_{\mathcal{A}}$, and $(A,G)$ a nested face of $\mathcal{A}$. Then, 
	\begin{equation} \label{EqF1}
	\sum_{F \in F_{\mathcal{A}}^{(A,G)}} (-1)^{\mathrm{rk}\,F} \sum_{\substack{L \in F_{\mathcal{A}} \\ FL \preceq G}} x_L \, = \, (-1)^{\mathrm{rk}\,G} \sum_{\substack{L \in F_{\mathcal{A}} \\ AL = \tilde{G}_A}} x_L,
	\end{equation}
	\begin{equation} \label{EqF2}
	\sum_{F \in F_{\mathcal{A}}^{(A,G)}} (-1)^{\mathrm{rk}\,F} \sum_{\substack{L \in F_{\mathcal{A}} \\ FL = G}} x_L \, = \, (-1)^{\mathrm{rk}\,G} \sum_{\substack{L \in F_{\mathcal{A}} \\ AL \preceq \tilde{G}_A}} x_L.
	\end{equation}
\end{corollary}

\begin{proof}
	We have $\displaystyle \sum_{F \in F_{\mathcal{A}}^{(A,G)}} (-1)^{\mathrm{rk}\,F} \sum_{\substack{L \in F_{\mathcal{A}} \\ FL \preceq G}} x_L = \sum_{L \in F_{\mathcal{A}}} \Big( \sum_{\substack{F \in F_{\mathcal{A}}^{(A,G)} \\ FL \preceq G}} (-1)^{\mathrm{rk}\,F} \Big) x_L$. As the condition $FL \leq G$ is equivalent to $FLG = G$, using Proposition~\ref{Witt} with $D=G$ and $C=LG$, we get
	$$\sum_{L \in F_{\mathcal{A}}} \Big( \sum_{\substack{F \in F_{\mathcal{A}}^{(A,G)} \\ FL \preceq G}} (-1)^{\mathrm{rk}\,F} \Big) x_L = (-1)^{\mathrm{rk}\,G} \sum_{\substack{L \in F_{\mathcal{A}} \\ ALG = \tilde{G}_A}} x_L = (-1)^{\mathrm{rk}\,G} \sum_{\substack{L \in F_{\mathcal{A}} \\ AL = \tilde{G}_A}} x_L.$$
	Similarly, $\displaystyle \sum_{F \in F_{\mathcal{A}}^{(A,G)}} (-1)^{\mathrm{rk}\,F} \sum_{\substack{L \in F_{\mathcal{A}} \\ FL = G}} x_L = \sum_{L \in F_{\mathcal{A}}} \Big( \sum_{\substack{F \in F_{\mathcal{A}}^{(A,G)} \\ FL = G}} (-1)^{\mathrm{rk}\,F} \Big) x_L$. As the condition $FL = G$ is equivalent to $FL\tilde{G}_A = \tilde{G}_A$ for every $A \in \{F \in \min F_{\mathcal{A}}\ |\ F \preceq G\}$. Using Proposition~\ref{Witt} with $D=G$ and $C=L\tilde{G}_A$, we get
	$$\sum_{L \in F_{\mathcal{A}}} \Big( \sum_{\substack{F \in F_{\mathcal{A}}^{(A,G)} \\ FL = G}} (-1)^{\mathrm{rk}\,F} \Big) x_L = (-1)^{\mathrm{rk}\,G} \sum_{\substack{L \in F_{\mathcal{A}} \\ AL\tilde{G}_A = \tilde{G}_A}} x_L = (-1)^{\mathrm{rk}\,G} \sum_{\substack{L \in F_{\mathcal{A}} \\ AL \preceq \tilde{G}_A}} x_L.$$
\end{proof}

\noindent Denote by $\breve{C}_{\mathcal{A}}$ the set formed by the bounded chambers of a dehyperplane arrangement $\mathcal{A}$. And the set of faces composing the closure of a chamber $D \in C_{\mathcal{A}}$ is $F_{\overline{D}} := \{F \in F_{\mathcal{A}}\ |\ F \preceq D\}$.

\begin{proposition} \label{PrD} 
Let $\mathcal{A}$ be a dehyperplane arrangement in $T \in \mathscr{R}_n$, and assume $\breve{C}_{\mathcal{A}} \neq \emptyset$. Then, $$\forall D \in \breve{C}_{\mathcal{A}}:\ \sum_{F \in F_{\overline{D}}} (-1)^{\mathrm{rk}\,F} \sum_{\substack{C \in C_{\mathcal{A}} \\ FC = D}} x_C = (-1)^{c_{\mathcal{A}}} x_D.$$
\end{proposition}

\begin{proof}
Let $\{F_i\}_{i \in I} \varsubsetneq F_{\mathcal{A}}^{n-1,D}$ such that, if $\#I>1$, then
$$\forall i \in I,\, \exists j \in I \setminus \{i\}:\ \dim \overline{F_i} \cap \overline{F_j} = n-2.$$ 
We will use $\displaystyle \chi\Big(\overline{D} \setminus \bigcup_{i \in I} \overline{F_i}\Big) = \chi(\overline{D}) - \chi\Big(\bigcup_{i \in I} \overline{F_i}\Big) = 0$ to prove Proposition~\ref{PrD}. Consider now
$$\displaystyle \sum_{F \in F_{\overline{D}}} (-1)^{\mathrm{rk}\,F} \sum_{\substack{C \in C_{\mathcal{A}} \\ FC = D}} x_C = \sum_{C \in C_{\mathcal{A}}} \Big( \sum_{\substack{F \in F_{\overline{D}} \\ FC = D}} (-1)^{\mathrm{rk}\,F} \Big) x_C.$$
If $C \neq D$, define the dehyperplane arrangement $\mathcal{A}_{C,D} := \{H \in \mathcal{A}\ |\ \epsilon_{H}(C) \neq \epsilon_{H}(D)\}$. Remark that if $\#\mathcal{A}_{C,D} > 1$, then $$\forall H \in \mathcal{A}_{C,D},\, \exists H' \in \mathcal{A}_{C,D} \setminus \{H\}:\ \dim \overline{H} \cap \overline{H'} = n-2.$$ We obtain
$$\sum_{\substack{F \in F_{\overline{D}} \\ FC = D}} (-1)^{\mathrm{rk}\,F} = (-1)^{c_{\mathcal{A}}} \sum_{\substack{F \in F_{\overline{D}} \\ \forall H \in \mathcal{A}_{C,D}:\ \epsilon_H(F) \neq 0}} (-1)^{\dim F} = (-1)^{c_{\mathcal{A}}} \chi\Big(\overline{D} \setminus \bigcup_{H \in \mathcal{A}_{C,D}} \overline{D} \cap H\Big) = 0.$$
If $C = D$, then $\displaystyle \sum_{\substack{F \in F_{\overline{D}} \\ FC = D}} (-1)^{\mathrm{rk}\,F} = (-1)^{c_{\mathcal{A}}} \chi(\overline{D}) = (-1)^{c_{\mathcal{A}}}$.
\end{proof}

\section{Proof of Theorem~\ref{ThDet}} \label{SeDet}

\noindent We compute $\det V_{\mathcal{A}}^K$, and justify the definition of a face multiplicity.

\begin{lemma} \label{CFD}
Let $\mathcal{A}$ be a dehyperplane arrangement in $T \in \mathscr{R}_n$, $C, D \in C_{\mathcal{A}}$, and $F \in F_{\mathcal{A}}$ such that $F \preceq C$. Then, $$\mathrm{v}(C,D) = \mathrm{v}(C,FD) \, \mathrm{v}(FD,D).$$
\end{lemma}

\begin{proof} We know that $FD$ is a chamber in $C_{\mathcal{A}}^K = \{E \cap K\ |\ E \in C_{\mathcal{A}_F}\}$ with $\displaystyle K = \bigcap_{\substack{H \in \mathcal{A} \\ \epsilon_H(F) \neq 0}} H^{\epsilon_H(F)}$.
\begin{itemize}
\item If $FD = C$, then $\mathscr{H}(C,FD) \sqcup \mathscr{H}(FD,D) = \mathscr{H}(C,C) \sqcup \mathscr{H}(C,D) = \mathscr{H}(C,D)$.
\item Else, $\mathscr{H}(C,FD) \sqcup \mathscr{H}(FD,D)$ is equal to
$$\big\{P^{\epsilon_P(C)}\ \big|\ P \in \mathcal{A}_F,\, \epsilon_P(C) \neq \epsilon_P(FD)\big\} \sqcup \big\{Q^{\epsilon_Q(FD)}\ \big|\ Q \in \mathcal{A} \setminus \mathcal{A}_F,\, \epsilon_Q(FD) \neq \epsilon_Q(D)\big\}$$ which is $\mathscr{H}(C,D)$ since $\epsilon_Q(FD) = \epsilon_Q(C)$ for every $Q \in \mathcal{A} \setminus \mathcal{A}_F$.
\end{itemize}
\end{proof}

\noindent For a dehyperplane arrangement $\mathcal{A}$ in $T$, $\displaystyle M_{\mathcal{A}} := \Big\{\sum_{C \in C_{\mathcal{A}}} x_C C\ \Big|\ x_C \in R_{\mathcal{A}}\Big\}$ is the module of $R_{\mathcal{A}}$-linear combinations of chambers in $C_{\mathcal{A}}$. Let $\{C^*\}_{C \in C_{\mathcal{A}}}$ be the dual basis of the basis $C_{\mathcal{A}}$ of $M_{\mathcal{A}}$. Define the linear map $\gamma_{\mathcal{A}}: M_{\mathcal{A}} \rightarrow M_{\mathcal{A}}^*$, for $D \in C_{\mathcal{A}}$, by $\displaystyle \gamma_{\mathcal{A}}(D) := \sum_{C \in C_{\mathcal{A}}} \mathrm{v}(D,C)\, C^*$. For a nested face $(A,D)$ with $D \in C_{\mathcal{A}}$, let $\displaystyle \mathrm{m}(A,D) := \sum_{\substack{C \in C_{\mathcal{A}} \\ AC = D}} \mathrm{v}(D,C)\, C^* \in M_{\mathcal{A}}^*$.

\noindent Define the extension ring $\displaystyle B_{\mathcal{A}} := \bigg\{\frac{p}{\displaystyle \prod_{F \in F_{\mathcal{A}}  \setminus C_{\mathcal{A}}} (1 - \mathrm{b}_F)^{k_F}}\ \bigg|\ p \in R_{\mathcal{A}},\, k_F \in \mathbb{N}\bigg\}$ of $R_{\mathcal{A}}$.

\begin{proposition} \label{mAD}
Let $\mathcal{A}$ be a dehyperplane arrangement in $T \in \mathscr{R}_n$, $D \in C_{\mathcal{A}}$, and $(A,D)$ a nested face of $\mathcal{A}$. Then, $$\mathrm{m}(A,D) = \sum_{C \in C_{\mathcal{A}}} x_C \, \gamma_{\mathcal{A}}(C) \quad \text{with} \quad x_C \in B_{\mathcal{A}}.$$
\end{proposition}

\begin{proof}
The proof is inspired from the backward induction in the proof of \cite[Proposition~8.13]{AgMa}. We obviously have $\mathrm{m}(D,D) = \gamma_{\mathcal{A}}(D)$. Then, Proposition~\ref{Witt} applied to $x_C = \mathrm{v}(D,C)\, C^*$ in addition to Lemma \ref{CFD} yield
$$\sum_{F \in F_{\mathcal{A}}^{(A,D)}} (-1)^{\mathrm{rk}\,F} \mathrm{m}(F,D) = (-1)^{\mathrm{rk}\,D} \sum_{\substack{C \in C_{\mathcal{A}} \\ AC = \tilde{D}_A}} \mathrm{v}(D,C)\, C^* = (-1)^{\mathrm{rk}\,D} \, \mathrm{v}(D,\tilde{D}_A) \, \mathrm{m}(A,\tilde{D}_A).$$
Hence, $\displaystyle \mathrm{m}(A,D) - (-1)^{\mathrm{rk}\,D - \mathrm{rk}\,A} \, \mathrm{v}(D,\tilde{D}_A) \, \mathrm{m}(A,\tilde{D}_A) = \sum_{F \in F_{\mathcal{A}}^{(A,D)} \setminus \{A\}} (-1)^{\mathrm{rk}\,F - \mathrm{rk}\,A +1} \mathrm{m}(F,D)$.
By induction hypothesis, for every $C \in C_{\mathcal{A}}$, there exists $a_C \in B_{\mathcal{A}}$, such that $$\sum_{F \in F_{\mathcal{A}}^{(A,D)} \setminus \{A\}} (-1)^{\mathrm{rk}\,F - \mathrm{rk}\,A +1} \mathrm{m}(F,D) = \sum_{C \in C_{\mathcal{A}}} a_C \, \gamma_{\mathcal{A}}(C).$$
Since $A \preceq \tilde{D}_A$ and $(\widetilde{\tilde{D}_A})_A = D$, replacing $D$ with $A\tilde{D}$, there exists also $e_C \in B_{\mathcal{A}}$ for every $C \in C_{\mathcal{A}}$ such that $\displaystyle \mathrm{m}(A,\tilde{D}_A) - (-1)^{\mathrm{rk}\,\tilde{D}_A - \mathrm{rk}\,A} \, v(\tilde{D}_A,D) \, \mathrm{m}(A,D) = \sum_{C \in C_{\mathcal{A}}} e_C \, \gamma_{\mathcal{A}}(C)$. Therefore, $$\displaystyle \mathrm{m}(A,D) = \sum_{C \in C_{\mathcal{A}}} \frac{a_C + (-1)^{\mathrm{rk}\,D - \mathrm{rk}\,A} \, \mathrm{v}(D,\tilde{D}_A) \, e_C}{1 - \mathrm{b}_A} \gamma_{\mathcal{A}}(C).$$
\end{proof}

\noindent A bounded chamber $C$ of a dehyperplane arrangement $\mathcal{A}$ is locally bounded in $K$ if $\partial C \varsubsetneq K$. Denote by $\breve{C}_{\mathcal{A}}^K$ the set formed by the chambers of $\mathcal{A}$ which are locally bounded in $K$.

\begin{definition}
	Let $\mathcal{A}$ be a dehyperplane arrangement in $T \in \mathscr{R}_n$, $K \in K_{\mathcal{A}}$, and $D$ a chamber in $C_{\mathcal{A}}^K \setminus \breve{C}_{\mathcal{A}}^K$. The \textbf{local boundering} of $D$ to a locally bounded chamber $D'$ in $K$ consists on inserting a minimal number of dehyperplanes $H_1, \dots, H_k$ in $\mathcal{A}$ such that
	\begin{itemize}
		\item $\mathcal{A}' = \mathcal{A} \sqcup \{H_i\}_{i \in [k]}$ is a dehyperplane arrangement in $T$, and obviously $K \in K_{\mathcal{A}'}$,
		\item $H_i$ divides $D$ into two chambers $D_i, D_i'$ such that $D_i \varsubsetneq H_i^+$ and $\chi(\overline{D_i'} \cap K)=0$,
		\item if $\displaystyle L = \bigcap_{i \in [k]}H_i^+$ and $C_{\mathcal{A}'}^K(L) := \{C \cap L\ |\ C \in C_{\mathcal{A}}^K\}$, there is a bijection $\mathrm{g}_K: C_{\mathcal{A}'}^K(L) \rightarrow C_{\mathcal{A}}^K$ such that, for every $C \in C_{\mathcal{A}'}^K(L)$, we have $\epsilon_{\mathcal{A}}(C) = \epsilon_{\mathcal{A}}\big(\mathrm{g}_K(C)\big)$,
		\item $\displaystyle D' = \bigcap_{i \in [k]} D_i$ is the chamber of $\mathcal{A}'$ such that $\epsilon_{\mathcal{A}}(D') = \epsilon_{\mathcal{A}}(D)$ and $\breve{C}_{\mathcal{A}'}^K = \breve{C}_{\mathcal{A}}^K \sqcup \{D'\}$.
	\end{itemize}
\end{definition}

\begin{example}
	Consider the dehyperplane arrangement $\mathcal{A}_{ex}$ of Figure~\ref{Ex1}. In Figure~\ref{Ex2}, we have a boundering of the chamber $-+++$ with the dehyperplanes $Q_1, Q_2$, and a boundering of the chamber $++++$ with the dehyperplanes $Q_1, Q_2, Q_3, Q_4$.
	
	\begin{figure}[h]
		\centering
		\includegraphics[scale=0.73]{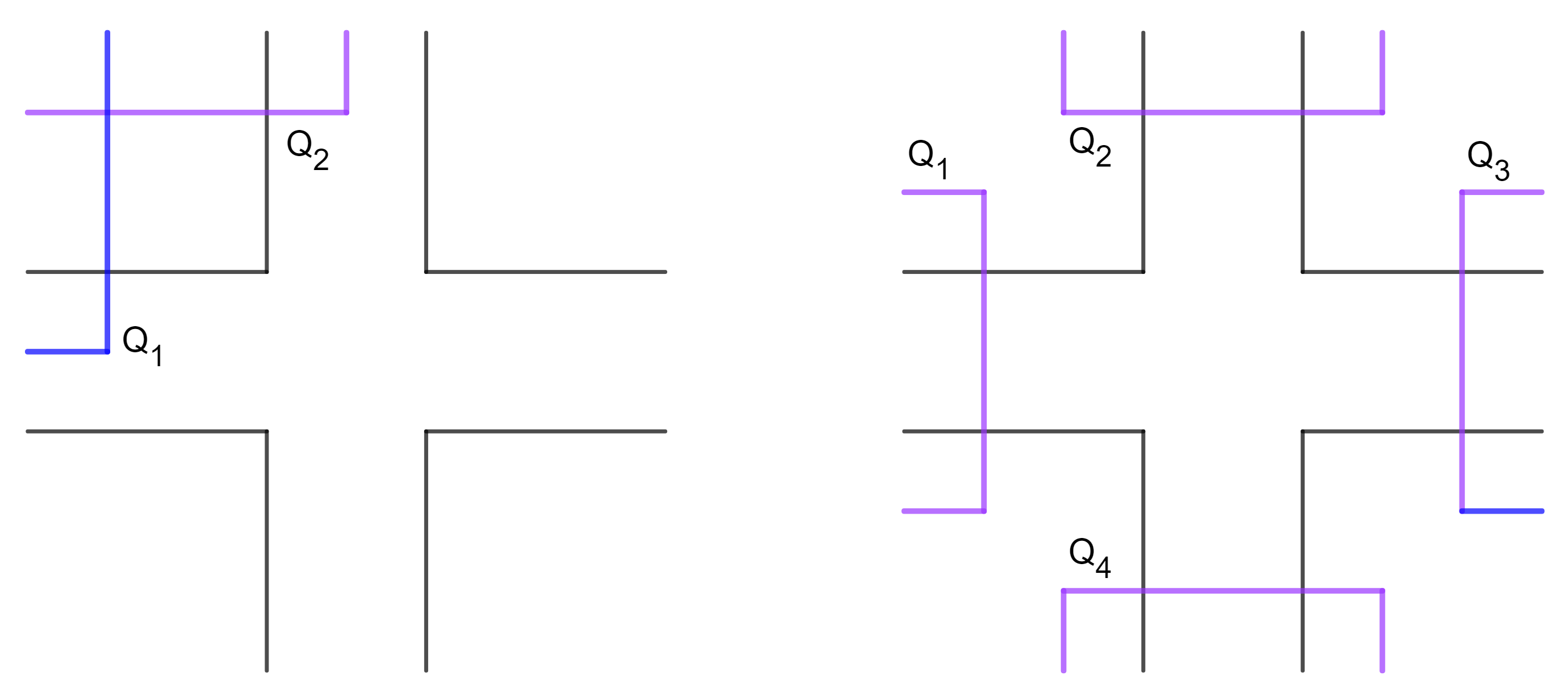}
		\caption{Boundering of the Chambers $-+++$ and $++++$ of $\mathcal{A}_{ex}$}
		\label{Ex2}
	\end{figure}
\end{example}

\noindent Let $B_{\mathcal{A}}^K$ be the subring $\displaystyle \Bigg\{\frac{p}{\displaystyle \prod_{F \in F_{\mathcal{A}}^K \setminus C_{\mathcal{A}}^K} (1 - \mathrm{b}_F)^{k_F}}\ \Bigg|\ p \in R_{\mathcal{A}},\, k_F \in \mathbb{N}\Bigg\}$ of $B_{\mathcal{A}}$, and $M_{\mathcal{A}}^K$ the submodule $\displaystyle \Big\{\sum_{C \in C_{\mathcal{A}}^K} x_C C\ \Big|\ x_C \in R_{\mathcal{A}}\Big\}$ of $M_{\mathcal{A}}$ for an apartment $K$ of a dehyperplane arrangement $\mathcal{A}$. Moreover, define the linear map $\gamma_{\mathcal{A}}^K: M_{\mathcal{A}} \rightarrow M_{\mathcal{A}}^*$, for $D \in C_{\mathcal{A}}^K$, by $\displaystyle \gamma_{\mathcal{A}}^K(D) := \sum_{C \in C_{\mathcal{A}}^K} \mathrm{v}(D,C)\, C^*$, and for a nested face $(A,D)$ with $A \in F_{\mathcal{A}}^K$, let $\displaystyle \mathrm{m}_K(A,D) := \sum_{\substack{C \in C_{\mathcal{A}}^K \\ AC = D}} \mathrm{v}(D,C)\, C^* \in M_{\mathcal{A}}^*$.

\begin{theorem} \label{ThDK*}
Let $\mathcal{A}$ be a dehyperplane arrangement in $T \in \mathscr{R}_n$, $K \in K_{\mathcal{A}}$, and $D \in C_{\mathcal{A}}^K$. Then,
$$D^* = \sum_{C \in C_{\mathcal{A}}^K} x_C \, \gamma_{\mathcal{A}}^K(C) \quad \text{with} \quad x_C \in B_{\mathcal{A}}^K.$$
\end{theorem}

\begin{proof}
Let $\mathcal{K}$ be the subset containing the dehyperplanes $H \in \mathcal{A}$ such that $\dim H \cap \overline{K} = n-1$. Setting $q_H^+ = q_H^- = 0$ for every $H \in \mathcal{K}$, we get $\mathrm{v}(D,C) = 0$ whenever one of $C$ or $D$ is a chamber in $C_{\mathcal{A}}^K$ but the other not. Then if $D \in C_{\mathcal{A}}^K$, $$\displaystyle \gamma_{\mathcal{A}}(D) = \sum_{C \in C_{\mathcal{A}}^K} \mathrm{v}(D,C)\, C^* = \gamma_{\mathcal{A}}^K(D) \quad \text{and} \quad \mathrm{m}(A,D) = \sum_{\substack{C \in C_{\mathcal{A}}^K \\ AC = D}} \mathrm{v}(D,C)\, C^* = \mathrm{m}_K(A,D).$$

\noindent Suppose that $D \in \breve{C}_{\mathcal{A}}^K$. Applying $x_C = \mathrm{v}(D,C)\, C^*$ to Proposition~\ref{PrD}, we obtain $$\sum_{F \in F_{\overline{D}}} (-1)^{\mathrm{rk}\,F} \mathrm{m}_K(F,D) = (-1)^{c_{\mathcal{A}}} D^*.$$
From Proposition \ref{mAD}, we conclude that $\displaystyle D^* = \sum_{C \in C_{\mathcal{A}}} x_C \, \gamma_{\mathcal{A}}^K(C)$ with $x_C \in B_{\mathcal{A}}^K$.
	
\noindent Suppose now that $D \in C_{\mathcal{A}}^K \setminus \breve{C}_{\mathcal{A}}^K$. Consider the dehyperplane arrangement $\mathcal{A}' = \mathcal{A} \sqcup \{H_i\}_{i \in [k]}$ obtained from the local boundering of $D$ to $D'$, the apartment $\displaystyle L = \bigcap_{i \in [k]}H_i^+$, and the bijection $\mathrm{g}_K: C_{\mathcal{A}'}^K(L) \rightarrow C_{\mathcal{A}}^K$ with $\epsilon_{\mathcal{A}}(C) = \epsilon_{\mathcal{A}}\big(\mathrm{g}_K(C)\big)$. As $D' \in \breve{C}_{\mathcal{A}'}^K$, then $\displaystyle {D'}^* = \sum_{C \in C_{\mathcal{A}'}^K} x_C \, \gamma_{\mathcal{A}'}^K(C)$ where $x_C \in B_{\mathcal{A}'}^K$. Hence,
\begin{align*}
& {D'}^* = \sum_{C \in C_{\mathcal{A}'}^K(L)} x_C \, \gamma_{\mathcal{A}'}^K(C) + \sum_{C' \in C_{\mathcal{A}'}^K \setminus C_{\mathcal{A}'}^K(L)} x_{C'} \, \gamma_{\mathcal{A}'}^K(C') \\
& {D'}^* - \sum_{C \in C_{\mathcal{A}'}^K(L)} x_C \, \gamma_{\mathcal{A}'}^K(C) = \sum_{C' \in C_{\mathcal{A}'}^K \setminus C_{\mathcal{A}'}^K(L)} x_{C'} \, \gamma_{\mathcal{A}'}^K(C').
\end{align*}
Setting $q_{H_i}^+ = q_{H_i}^- = 0$ for $i \in [k]$, we get on one side
$$\displaystyle {D'}^* - \sum_{C \in C_{\mathcal{A}'}^K(L)} x_C \, \gamma_{\mathcal{A}'}^K(C)\, \in \, \bigg\{\sum_{C \in C_{\mathcal{A}'}^K(L)} x_{C^*} C^*\ \bigg|\ x_{C^*} \in B_{\mathcal{A}}^K\bigg\},$$ and $\displaystyle \sum_{C' \in C_{\mathcal{A}'}^K \setminus C_{\mathcal{A}'}^K(L)} x_{C'} \, \gamma_{\mathcal{A}'}^K(C')\, \in \, \bigg\{\sum_{C' \in C_{\mathcal{A}'}^K \setminus C_{\mathcal{A}'}^K(L)} x_{{C'}^*} {C'}^*\ \bigg|\ x_{{C'}^*} \in B_{\mathcal{A}}^K\bigg\}$ on the other side. The only possibility is
$\displaystyle {D'}^* - \sum_{C \in C_{\mathcal{A}'}^K(L)} x_C \, \gamma_{\mathcal{A}'}^K(C)\ =\ \sum_{C' \in C_{\mathcal{A}'}^K \setminus C_{\mathcal{A}'}^K(L)} x_{C'} \, \gamma_{\mathcal{A}'}^K(C')\ =\ 0$. Finally, replacing $C$ by $\mathrm{g}_K(C)$ for every $C \in C_{\mathcal{A}'}^K(L)$, we conclude that $\displaystyle D^* = \sum_{C \in C_{\mathcal{A}}^K} x_C \, \gamma_{\mathcal{A}}^K(C)$.
\end{proof}

\begin{proposition} \label{PrDet}
Let $\mathcal{A}$ a dehyperplane arrangement in $T \in \mathscr{R}_n$, and $K \in K_{\mathcal{A}}$. To every face $F \in F_{\mathcal{A}}^K \setminus C_{\mathcal{A}}^K$ can be associated a nonnegative integer $l_F$ such that $$\det V_{\mathcal{A}}^K = \prod_{F \in F_{\mathcal{A}}^K \setminus C_{\mathcal{A}}^K} (1 - \mathrm{b}_F)^{l_F}.$$ 
\end{proposition}

\begin{proof}
We extend the proof of \cite[Proposition~5.4]{Ra1} to apartment of dehyperplane arrangements, namely, we first note that $V_{\mathcal{A}}^K$ is the matrix representation of $\gamma_{\mathcal{A}}^K$. Then, we know from Theorem~\ref{ThDK*} that $\displaystyle {(\gamma_{\mathcal{A}}^K)}^{-1}(D^*) = \sum_{C \in C_{\mathcal{A}}^K} x_C \, C$ with $x_C \in B_{\mathcal{A}}^K$ for $D \in C_{\mathcal{A}}^K$. Hence, each entry of ${(V_{\mathcal{A}}^K)}^{-1}$ is an element of $B_{\mathcal{A}}^K$. Since $\displaystyle {(V_{\mathcal{A}}^K)}^{-1} = \frac{\mathrm{adj}(V_{\mathcal{A}}^K)}{\det V_{\mathcal{A}}^K}$, the determinant of $V_{\mathcal{A}}^K$ has then the form $\displaystyle k \prod_{F \in F_{\mathcal{A}}^K \setminus C_{\mathcal{A}}^K} (1 - \mathrm{b}_F)^{l_F}$ with $k \in \mathbb{Z}$. As the constant term of $\det V_{\mathcal{A}}^K$ is $1$, we deduce that $k=1$.
\end{proof}

\noindent For $F \in F_\mathcal{A}$ and $H \in \mathcal{A}_F$, define the integer $\displaystyle \beta_F^H := \frac{\#\{C \in C_{\mathcal{A}}\ |\ \overline{C} \cap H = F\}}{2}$. The following proof not only proves Theorem~\ref{ThDet}, but also shows that, for any hyperplane in $\mathcal{A}_F$, $\beta_F^H$ is the same, which justifies the definition of the multiplicity.

\begin{proof}
From Proposition~\ref{PrDet}, we have $\displaystyle \det V_{\mathcal{A}}^K = \prod_{F \in F_{\mathcal{A}}^K \setminus C_{\mathcal{A}}^K} (1 - \mathrm{b}_F)^{l_F}$. Take a face $E \in F_{\mathcal{A}}^K \setminus C_{\mathcal{A}}^K$: there exists an apartment $L \in K_{\mathcal{A}}$ such that $E \subseteq L \subseteq K$, $\displaystyle \bigcap_{\substack{H \in \mathcal{A} \\ H \cap L \neq \emptyset}}H = E$, and
$$\det V_{\mathcal{A}}^L = \prod_{F \in F_{\mathcal{A}}^L \setminus C_{\mathcal{A}}^L} (1 - \mathrm{b}_F)^{l_F'}.$$ Setting $h_H^+ = h_H^- = 0$ for every $H \in \mathcal{A} \setminus \mathcal{A}_E$, we see that, for every $F \in F_{\mathcal{A}}^L \setminus C_{\mathcal{A}}^L$, $l_F = l_F'$. 

\noindent We prove by backward induction on the dimension of $E$ that $$\forall H,H' \in \mathcal{A}_E:\ \beta_E^H = \beta_E^{H'} = \beta_E \quad \text{and} \quad \det V_{\mathcal{A}}^L = \prod_{F \in F_{\mathcal{A}}^L \setminus C_{\mathcal{A}}^L} (1 - \mathrm{b}_F)^{\beta_F}.$$

\noindent Remark that $\displaystyle \beta_F^H = \frac{\#\{C \in C_{\mathcal{A}}^L\ |\ \overline{C} \cap H = F\}}{2}$. It is clear that, if $\dim E = n-1$, then $\beta_E = 1$ and $\det V_{\mathcal{A}}^L = 1 - \mathrm{b}_E$. If $\dim E < n-1$, by induction hypothesis, $$\det V_{\mathcal{A}}^L = (1 - \mathrm{b}_E)^{l_E} \, \prod_{F \, \in \, (F_{\mathcal{A}}^L \setminus C_{\mathcal{A}}^L) \setminus \{E\}} (1 - \mathrm{b}_F)^{\beta_F}.$$
Note that the leading monomial in $\det V_{\mathcal{A}}^L$ is $\displaystyle (-1)^{\frac{\#C_{\mathcal{A}}^L}{2}} \prod_{C \in C_{\mathcal{A}}^L} \mathrm{v}(C,\tilde{C}_E) \, = \, \big( -\prod_{H \in \mathcal{A}_E} h_H^+ \, h_H^- \big)^{\frac{\#C_{\mathcal{A}}^L}{2}}$.
Comparing the exponent of $h_H^+ \, h_H^-$, we get $\displaystyle l_E \, = \, \frac{\#C_{\mathcal{A}}^L}{2} - \sum_{\substack{F \, \in \, (F_{\mathcal{A}}^L \setminus C_{\mathcal{A}}^L) \setminus \{E\} \\ F \subseteq H}} \beta_F^H \, = \, \beta_E^H$.
\end{proof}

\section{Proof of Theorem~\ref{ThSy}} \label{SeSy}

\noindent We determine the solution space dimension of the Aguiar-Mahajan system, and solve that latter for central dehyperplane arrangements.

\noindent The Varchenko matrix of $\mathcal{A}^X$ is $V_{\mathcal{A}^X} :=\big|\mathrm{v}(D,C)\big|_{C,D \in C_{\mathcal{A}^X}}$. The centralization to a face $F \in F_{\mathcal{A}^X} \setminus C_{\mathcal{A}^X}$ is the dehyperplane arrangement $\mathcal{A}_F^X := \{H \in \mathcal{A}^X\ |\ F \subseteq H\}$ in $X$. The weight and multiplicity of $F$ in $X$ are respectively the monomial and integer
$$\mathrm{b}_F^X := \prod_{\substack{H \in \mathcal{A} \\ H \cap X \in \mathcal{A}_F^X}} q_H^+ q_H^- \quad \text{and} \quad \beta_F^X := \frac{\#\{C \in C_{\mathcal{A}^X}\ |\ \overline{C} \cap P = F\}}{2},$$
where $P \in \mathcal{A}_F^X$, and $\beta_F^X$ is independent of $P$ like the multiplicity of a dehyperplane in $\mathcal{A}$.

\begin{corollary} \label{CoVX}
Let $\mathcal{A}$ be a dehyperplane arrangement in $T \in \mathscr{R}_n$, and $X \in L_{\mathcal{A}}$. Then, $$\det V_{\mathcal{A}^X} = \prod_{F \in F_{\mathcal{A}^X} \setminus C_{\mathcal{A}^X}} (1 - \mathrm{b}_F^X)^{\beta_F^X}.$$
\end{corollary}

\begin{proof}
It is Corollary~\ref{CoTh} but for the dehyperplane arrangement $\mathcal{A}^X$.
\end{proof}

\noindent Define the assembly of Varchenko matrices $S_{\mathcal{A}} := (s_{F,G})_{F,G \in F_{\mathcal{A}}}$ by
$$s_{F,G} := \begin{cases}
\mathbf{v}(F,G) & \text{if}\ GF = G,\\
0 & \text{otherwise}
\end{cases}.$$

\begin{proposition} \label{PrAs}
Let $\mathcal{A}$ be a dehyperplane arrangement in $T \in \mathscr{R}_n$. Then,
$$\det S_{\mathcal{A}} = \prod_{X \in L_{\mathcal{A}}} \, \prod_{F \in F_{\mathcal{A}^X} \setminus C_{\mathcal{A}^X}} (1 - \mathrm{b}_F^X)^{\beta_F^X}.$$
\end{proposition}

\begin{proof}
We basically take up the argument in \cite[§~8.4.5]{AgMa} in a dehyperplane arrangement context. Write $S_{\mathcal{A}}$ as a block matrix indexed by flats, with the $(X,Y)$-block consisting of the entries $s_{F,G}$ such that $\mathrm{s}(F) = X$ and $\mathrm{s}(G) = Y$. Moreover, order the flats so that row $X$ appears above row $Y$ if $X < Y$. That block matrix is lower triangular with the diagonal block $(X,X)$-block being $V_{\mathcal{A}}^X$. Hence, $\displaystyle \det S_{\mathcal{A}} = \prod_{X \in L_{\mathcal{A}}} \det V_{\mathcal{A}}^X$, and it remains to apply Corollary~\ref{CoVX}.  
\end{proof}

\begin{lemma} \label{LeFG}
	Let $\mathcal{A}$ be a dehyperplane arrangement in $T \in \mathscr{R}_n$, and $F,G,L \in F_{\mathcal{A}}$ such that $FG \preceq L$. Then,
	$$\mathbf{v}(F,G) = \mathbf{v}(L,G) \quad \text{and} \quad \mathbf{v}(G,F) = \mathbf{v}(G,L).$$
\end{lemma}

\begin{proof}
	For $F,G \in F_{\mathcal{A}}$, define the set $\mathrm{d}(F,G) := \{H \in \mathcal{A}\ |\ \epsilon_H(F) \neq 0,\, \epsilon_H(G) = 0\}$. We have $\mathbf{v}(F,G) = \mathbf{v}(FG,GF) = \mathbf{v}(L,GFL)$. Moreover, $FG \preceq L$ also implies $\epsilon_{\mathrm{d}(F,G)}(F) = \epsilon_{\mathrm{d}(F,G)}(L)$. Then, $\mathbf{v}(L,GFL) = \mathbf{v}(L,GL) = \mathbf{v}(L,G)$.
	
	\noindent Similarly, $\mathbf{v}(G,F) = \mathbf{v}(GF,FG) = \mathbf{v}(GFL,L) = \mathbf{v}(GL,L) = \mathbf{v}(G,L)$.
\end{proof}

\begin{lemma} \label{LeFLLF}
	Let $\mathcal{A}$ be a dehyperplane arrangement in $T \in \mathscr{R}_n$, and $F,G,L \in F_{\mathcal{A}}$ such that $G \preceq L$. Then,
	$$\mathbf{v}(F,L) = \mathbf{v}(F,GF) \, \mathbf{v}(GF,L).$$
\end{lemma}

\begin{proof}
	For $F,G \in F_{\mathcal{A}}$, define the set $\mathrm{e}(F,G) := \{H \in \mathcal{A}\ |\ \epsilon_H(F) \neq 0,\, \epsilon_H(G) \neq 0\}$.
	
	\noindent On one side, $\mathbf{v}(F,GF) = \mathbf{v}(FG,GF)$ and $$\mathscr{H}(FG,GF) = \big\{H^{\epsilon_H(F)}\ \big| \ H \in \mathrm{e}(F,G),\, \epsilon_H(F) \neq \epsilon_H(G)\big\}.$$
	On the other side, $\mathbf{v}(GF,L) = \mathbf{v}(GFL,LF)$ and $$\mathscr{H}(GFL,LF) = \big\{H^{\epsilon_H(F)}\ \big| \ H \in \mathrm{e}(F,L) \setminus \mathrm{e}(F,G),\, \epsilon_H(F) \neq \epsilon_H(L)\big\}.$$
	Hence, $\mathscr{H}(FG,GF) \sqcup \mathscr{H}(GFL,LF) = \mathscr{H}(FL,LF)$.
\end{proof}

\begin{lemma} \label{LeSy}
	Let $\mathcal{A}$ be a dehyperplane arrangement in $T \in \mathscr{R}_n$. The solution space of the Aguiar-Mahajan system of $\mathcal{A}$ coincides with that of the linear equation system
	$$\sum_{\substack{F \in F_{\mathcal{A}} \\ LF = G}} x_F\,\mathbf{v}(F,G) = 0 \quad \text{indexed by} \quad L,G \in F_{\mathcal{A}} \setminus \min F_{\mathcal{A}} \quad \text{with} \quad L \preceq G.$$ 
\end{lemma}

\begin{proof}
	Note that the Aguiar-Mahajan system is smaller than that of Lemma~\ref{LeSy}. So we need to show that any solution of the former also solves the latter. The proof is inspired from the backward induction of the proof of \cite[Lemma~8.18]{AgMa}. It is clear that the solutions coincide if $L=G$. Let $A \in F_{\mathcal{A}} \setminus \min F_{\mathcal{A}}$ with $A \preceq G$, and start with Equation~\ref{EqF2} by replacing $x_F$ with $x_F \, \mathbf{v}(F,G)$. By induction, $\displaystyle \sum_{L \in F_{\mathcal{A}}^{(A,G)}} (-1)^{\mathrm{rk}\,L} \sum_{\substack{F \in F_{\mathcal{A}} \\ LF = G}} x_F \, \mathbf{v}(F,G) = (-1)^{\mathrm{rk}\,A} \sum_{\substack{F \in F_{\mathcal{A}} \\ AF = G}} x_F \, \mathbf{v}(F,G)$, thus
	\begin{align*}
	(-1)^{\mathrm{rk}\,A} \sum_{\substack{F \in F_{\mathcal{A}} \\ AF = G}} x_F \, \mathbf{v}(F,G)
	& = (-1)^{\mathrm{rk}\,G} \sum_{\substack{F \in F_{\mathcal{A}} \\ AF \preceq \tilde{G}_A}} x_F \, \mathbf{v}(F,G) \\
	& = (-1)^{\mathrm{rk}\,G} \sum_{\substack{F \in F_{\mathcal{A}} \\ AF = \tilde{G}_A}} x_F \, \mathbf{v}(F,G) \quad \text{as $\mathrm{s}(AF) = \mathrm{s}(G) = \mathrm{s}(\tilde{G}_A)$} \\
	& = (-1)^{\mathrm{rk}\,G} \sum_{\substack{F \in F_{\mathcal{A}} \\ AF = \tilde{G}_A}} x_F \, \mathbf{v}(F,AF) \, \mathbf{v}(AF,G) \quad \text{using Lemma~\ref{LeFLLF}} \\
	& = (-1)^{\mathrm{rk}\,G} \mathbf{v}(\tilde{G}_A,G) \sum_{\substack{F \in F_{\mathcal{A}} \\ AF = \tilde{G}_A}} x_F \, \mathbf{v}(F,\tilde{G}_A)  \quad \text{using Lemma~\ref{LeFG}}.
	\end{align*}
	Interchanging the roles of $G$ and $\tilde{G}_A$ yields a similar identity. Combining both ones, we obtain $$\big(1- \mathbf{v}(G,\tilde{G}_A) \, \mathbf{v}(\tilde{G}_A,G)\big) \sum_{\substack{F \in F_{\mathcal{A}} \\ AF = G}} x_F \, \mathbf{v}(F,G) = 0.$$
\end{proof}

\noindent We can finally proceed to the proof of Theorem~\ref{ThSy}.

\begin{proof}
	To the Aguiar-Mahajan system add the equations
	$$x_F = \alpha_F \quad \text{for} \quad F \in \min F_{\mathcal{A}},\, \alpha_F \in R_{\mathcal{A}},$$
	where $\alpha_F$ is fixed but arbitrary. The matrix of that linear system is the assembly $S_{\mathcal{A}}$. We know form Proposition~\ref{PrAs} that $\det S_{\mathcal{A}}$ is invertible in $B_{\mathcal{A}}$, so that system has a unique solution. Hence, the solution space dimension of the Aguiar-Mahajan system of $\mathcal{A}$ is $\# \min F_{\mathcal{A}}$.
	
	\noindent Now assume $\mathcal{A}$ is central. Lemma~\ref{LeSy} allows to consider the linear system $\displaystyle \sum_{\substack{F \in F_{\mathcal{A}} \\ LF = G}} x_F\,\mathbf{v}(F,G) = 0$ to solve the Aguiar-Mahajan system. We successively have
	\begin{align*}
	\sum_{\substack{F \in F_{\mathcal{A}} \\ LF = G}} x_F\,\mathbf{v}(F,G) & = 0 \\
	\sum_{\substack{F \in F_{\mathcal{A}} \\ LF = N \\ N \preceq G}} x_F\,\mathbf{v}(F,N) & = 0 \\
	\sum_{\substack{F \in F_{\mathcal{A}} \\ LF \preceq G}} x_F\,\mathbf{v}(F,G) & = 0 \quad \text{using Lemma~\ref{LeFG}}.
	\end{align*}
	Applying Equation~\ref{EqF1} with $x_F$ replaced by $x_F \, \mathrm{v}(F,G)$ and $A$ by $O$, we obtain
	$$\sum_{\substack{F \in F_{\mathcal{A}} \\ F \prec G}} x_F + x_G = (-1)^{\mathrm{rk}\,G} x_{\tilde{G}}\,\mathbf{v}(\tilde{G},G).$$
	Interchanging the roles of $G$ and $\tilde{G}$ yields $\displaystyle \sum_{\substack{F \in F_{\mathcal{A}} \\ F \prec \tilde{G}}} x_F + x_{\tilde{G}} = (-1)^{\mathrm{rk}\,G} x_{G}\,\mathbf{v}(G,\tilde{G})$. Thus,
	$$\displaystyle x_G = -\sum_{\substack{F \in F_{\mathcal{A}} \\ F \prec G}} x_F + x_G \, \mathbf{v}(G,\tilde{G}) \, \mathbf{v}(\tilde{G},G) - (-1)^{\mathrm{rk}\,G} \, \mathbf{v}(\tilde{G},G) \sum_{\substack{F \in F_{\mathcal{A}} \\ F \prec \tilde{G}}} x_F.$$
\end{proof}

\bibliographystyle{abbrvnat}

\end{document}